\documentclass[12pt, a4paper]{amsart}

\pdfoutput=1

\usepackage{enumerate}
\usepackage[nocompress]{cite}
\usepackage{hyperref}
\usepackage{orcidlink}
\usepackage[margin=3.14cm]{geometry}
\usepackage{mathtools}
\usepackage{amsaddr}
\usepackage{amssymb}

\usepackage{microtype}

\usepackage{soul}
\setul{1.5pt}{0.2ex}
\renewcommand{\underline}[1]{\ul{#1}}

\theoremstyle{plain}
\newtheorem{theorem}{Theorem}[section]

\newtheorem{lemma}[theorem]{Lemma}

\theoremstyle{definition}
\newtheorem{definition}[theorem]{Definition}

\newtheorem{example}[theorem]{Example}

\newcommand		{\abs}[1]{\left|#1\right|}

\newcommand    	{\N} {\mathbb N}
\newcommand    	{\Z} {\mathbb Z}
\newcommand    	{\sym} {\mathrm {Sym}}
\renewcommand   {\O} {\mathcal O}

\usepackage{array}
\newcolumntype{L}[1]{>{\raggedright\let\newline\\\arraybackslash\hspace{0pt}}m{#1}}
\newcolumntype{C}[1]{>{\centering\let\newline\\\arraybackslash\hspace{0pt}}m{#1}}
\newcolumntype{R}[1]{>{\raggedleft\let\newline\\\arraybackslash\hspace{0pt}}m{#1}}

\title[The smartest person in the room is in the wrong room]
{If you are the smartest person in the room, \\ you are in the wrong room}
\author{Davide Sclosa\orcidlink{0000-0003-0806-2591}}
\address{University of Crete,
Department of Mathematics and Applied Mathematics, Heraklion, Greece}
\address{Vrije Universiteit Amsterdam,
Department of Mathematics, Amsterdam, \mbox{The Netherlands}}
\date{\today}

\begin{document}
\maketitle
\begin{abstract}
If taken seriously, the advice in the title leads to interesting combinatorics.
Consider $N$~people moving between $M$~rooms as follows:
at each step, simultaneously,
the smartest person in each room moves to a different room of their choice,
while no one else moves. The process repeats.
In this paper we determine which configurations are reachable, from which other configurations,
and provide bounds on the number of moves.
Namely,
let~$G(N,M)$ be the directed graph with vertices representing all~$M^N$ configurations and
edges representing possible moves.
We prove that the graph~$G(N,M)$ is weakly connected, and that it is strongly connected
if and only if~$M\geq N+1$ (one extra room for maneuvering is both required and sufficient).
For~$M\leq N$, we show that the graph has a giant strongly connected component
with~$\Theta(M^N)$ vertices and diameter~$\O(N^2)$.
\end{abstract}

\section{Introduction}
Fix~$N\geq 1$ and~$M\geq 2$. Denote~$[K]=\{1,\ldots,K\}$.
Let~$V=\{f:[N]\to[M]\}$ and let~$E$ be the set of pairs~$(f,g)\in V\times V$
with the following property:
for all~$k\in [N]$ we have~$g(k)\neq f(k)$ if and only if~$k = \max f^{-1}(f(k))$.
Define~$G(N,M) = (V,E)$.

Interpreted as in the abstract,
the total order~$1<\cdots<N$ represents smartness.
Each vertex is a configurations of~$N$ people in~$M$ rooms.
The property ``$g(k)\neq f(k)$ if and only if~$k = \max f^{-1}(f(k))$'' states that,
from the configuration~$f$ to~$g$, exactly those people who are the smartest in their
room under~$f$ change rooms.

We described the rules of the game.
The goal of the game is to go from one given configuration to another given configuration.
This is not always possible:
in Section~\ref{sec:proof_1} we prove the following theorem,
showing that having one extra room for maneuvering
is both necessary and sufficient.

\begin{theorem} \label{thm:1}
The graph~$G(N,M)$ is strongly connected if and only if~$M\geq N+1$.
\end{theorem}

The graph~$G(N,M)$ is always weakly connected.
The next theorem, proved in Section~\ref{sec:proof_2},
describes the strongly components in the case~$M\leq N$: the graph contains
a single giant component, and
several one-vertex components.
Every path enters the giant component in at most one step
and never leaves afterwords.

\begin{theorem} \label{thm:2}
Suppose that~$3\leq M \leq N$. Then~$G(N,M)$ is weakly connected, has~$1$ strongly
connected component containing~$M^N - M^{N-M+1}$ vertices,
and~$M^{N-M+1}$ strongly connected components
each containing a single vertex of in-degree zero.
\end{theorem}

\begin{figure}
\includegraphics[scale=0.5]{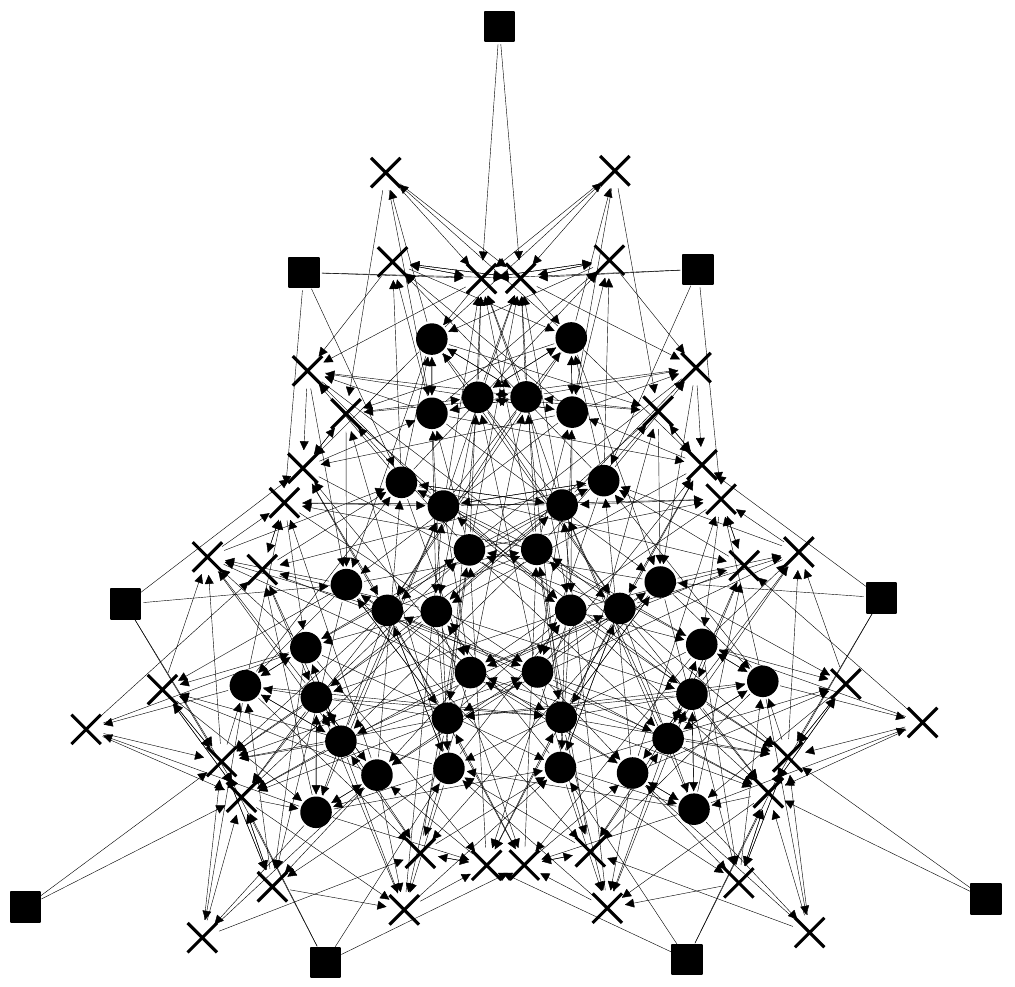}
\caption{The configuration graph~$G(N,M)$ for~$N=4$ people and~$M=3$ rooms.
The set~$V_s$ is labeled~$\bullet$,
the set~$V_c$ is labeled~{\tiny $\blacksquare$},
the other vertices are labeled~$\times$.
The vertices labeled~$\bullet$ and~$\times$ form a strongly connected component.
Vertices labeled~{\tiny $\blacksquare$} have no way to be reached, and form strong
components on their own.
}
\label{fig:1}
\end{figure}

As we will see, the ``unreachable'' configurations
of in-degree zero are exactly those
in which the $M$~smartest people share the same room.

The case~$2=M<N$, excluded from Theorem~\ref{thm:2},
is exceptional but easy to describe:
the two smartest people occupy both rooms in at most one step,
and swap places with no-one else moving afterwards.
In this case~$G(N,M)$ is the union of several small weakly connected components.


\subsection{Asynchronous Game}
Consider the variant of our game in which, at each step, a room is selected arbitrarily
and the smartest person of that room only moves to another room, while no one else moves.
This asynchronous game, which has the same priority rule as the original one, but in which
only one person moves at a time, will be analyzed in Section~\ref{sec:proof_2},
as a tool to prove Theorem~\ref{thm:2}.
We will see that
the configuration graph of the asynchronous game is always strongly connected:
the unreachable configurations
(the vertices with in-degree zero appearing in Theorem~\ref{thm:2})
disappear if we allow a single person to move at a time.
Unless stated explicitly, throughout the paper we will always refer to
the original synchronous version of the game.


\subsection{Relation to Towers of Hanoi}
The game discussed in this paper resembles the classical Towers of Hanoi
problem~\cite{romik2006shortest, cull1985towers, vsunic2012twin}.
There is a natural analogy between the two, with people corresponding to disks and
rooms corresponding to towers.
The key difference lies in the priority rule. The Towers of Hanoi
have a last-in first-out rule (only the most recently placed disk may
be moved from a tower). By contrast, our game ignores temporal order, and enforces
a fixed global ranking among the moving elements.
This difference has major consequences for worst-case complexity.
Towers of Hanoi requires an exponential number of moves, while
both the synchronous and asynchronous variants of our game only require
a polynomial number of moves.


\subsection{Complexity}
Fix the number of rooms~$M\geq 3$ and let the number of people~$N$ vary.
We will use standard asymptotic notation.\footnote{
Given two functions~$a,b:\N\to [0,\infty)$,
we denote~$a(N)=\O(b(N))$ if eventually~$a(N)\leq c\cdot b(N)$ for some
constant~$c>0$, we denote~$a(N)=\Omega(b(N))$ if eventually~$a(N)\geq c\cdot b(N)$
for some constant~$c>0$,
and~$a(N)=\Theta(b(N))$ if both~$a(N)=\O(b(N))$ and~$a(N)=\Omega(b(N))$ hold.}
The proof of Theorem~\ref{thm:1} shows that if~$N<M$, any configuration can be reached
from any other in~$\O(N)$ steps. For~$N\geq M$, the set~$V_c$ of configurations with
in-degree zero arises (Definition~\ref{def:Vc}).
The set~$V\setminus V_c$ is the giant component in Theorem~\ref{thm:2}.
The proof of Theorem~\ref{thm:2} shows that any configuration~$g\in V\setminus V_c$
can be reached from any other configuration~$f\in V$
in at most~$\O(N^2)$ steps.
Therefore, the diameter of the giant component grows at least linearly and
at most quadratically in~$N$. We are unable to determine the exact asymptotic growth,
see Footnote~\ref{foot} for more details.


\subsection{Random Walk Interpretation}
Suppose that the smartest person in each room independently chooses their next room
uniformly at random from among the other rooms.
We can then interpret the game as a random walk~\cite{lovasz1995mixing}
on the graph~$G(N,M)$, in which the next vertex is chosen uniformly at random
among the out-neighbors.
Interpreted this way,
Theorem~\ref{thm:1} and Theorem~\ref{thm:2} describe recurrence, and imply uniqueness
of the invariant measure.
Alternatively, we can interpret the game
as a multiple random walk, with $N$~random walkers moving among $M$~rooms
under a priority rule~\cite{cox1989coalescing}. The type of priority rule
we consider appears to be novel.



\subsection{Dynamical Systems Interpretation}
Let~$N,M\geq 3$.
Consider the operator that maps each set of vertices to its out-neighborhood.
Namely, let~$X=\{A \subseteq V: \ A\neq \emptyset\}$
and define~$T(A)=\{w\in V: \ (v,w)\in E,\ \text{for some } v\in A\}$.
We can then interpret our results dynamically.
The giant component~$V\setminus V_c$ is the only fixed point of the dynamical
system~$(X,T)$, and for all~$A\in X$ we have~$\lim_{n\to\infty} T^n(A)=V\setminus V_c$.


\section{Proof of Theorem~\ref{thm:1}} \label{sec:proof_1}
The proof of Theorem~\ref{thm:1} is broken down in a series of lemmas that will be
needed again in the proof of Theorem~\ref{thm:2}.
These lemmas are either impossibility arguments, or explicit algorithms.
Such algorithms are based on
the algebraic properties of \emph{derangements}, i.e., permutations with no fixed points.

As in the introduction, we assume~$N\geq 1$ and~$M\geq 2$
to avoid trivial cases
(for~$M=1$ the game is ill-defined since the smartest person
has nowhere to go).

Inspired by~\cite{tao2006additive}, we use brackets to denote integer intervals. For~$K,H\in \Z$ let
\[
	[K,H] = \{K,K+1,\ldots,H\} \cap \Z_{\geq 1}
\]
and~$(K,H] = [K,H]\setminus \{K\}$.
If~$1\leq K \leq H$ then~$[K,H] = \{K,K+1,\ldots,H\}$,
if~$K\leq 1 \leq H$ then~$[K,H] = [H]$.

We will describe a subset of vertices~$V_s\subseteq V$ that is particularly
well connected, and a subset of vertices~$V_c\subseteq V$ that is particularly poorly
connected.

\begin{definition}
Let~$V_s \subseteq V$ be the set of configurations~$f$
such that the restriction of~$f$ to~$(N-M,N]$ is injective.
\end{definition}

Intuitively, $V_s$~is the set of configurations in which the $M$~smartest people
among~$1,\ldots,N$ occupy different rooms;
if~$N \leq M$, this means that everyone is in a different room.
As we will see, the set~$V_s$ serves as a central hub in the graph~$G(N,M)$,
see Figure~\ref{fig:1}.

\begin{lemma} \label{lem:spread}
The set~$V_s$ can be reached from any configuration and in at most~$N$ steps.
\end{lemma}
\begin{proof}
Our goal is to place the~$M$ smartest people~$(N-M,N]$
into the $M$~different rooms.
To this end, movements of the people in~$[1,N-M]$
can be chosen arbitrarily, since they do not constrain the movements of~$(N-M,N]$.
Therefore, we may and will assume~$N\leq M$.

Fix any configuration~$f\in V$.
Let~$g$ be obtained from~$f$ as follows:
for each room with at least $2$~people move the smartest person to a distinct empty room
(which exists since~$N\leq M$),
and derange the set of people who occupy a room individually.
Then~$(f,g)$ is an edge.
Unless~$f\in V_s$, the configuration~$g$ occupies strictly more rooms.
By proceeding this way, we reach~$V_s$ in at most~$N$ steps.
\end{proof}

The following lemma shows that, with the exception of~$n=3$,
the symmetric group~$\sym(n)$ is generated by derangements.
The case~$n=3$ is indeed an exception since derangements of~$3$ objects are even permutations.

\begin{lemma} \label{lem:derangements}
For every integer~$n\neq 3$ the symmetric group~$\sym(n)$ is generated by derangements.
More precisely,
every permutation of~$n\neq 3$ objects is a product of at most~$4$ derangements.
\end{lemma}

\begin{proof}
The statement is vacuous for~$n=1$ and trivial for~$n=2$.
Let~$n\geq 4$.
Husemoller~\cite[Proposition 4]{husemoller1962ramified} attributes to Gleason
the result that every even permutation is a product of two $n$-cycles
(see~\cite{stanley1981factorization} for a more general result).
In particular, every even permutation is a product of two derangements.
Since every odd permutation is a product of an
even permutation and a transposition,
it remains to show that every transposition is a product of two derangements.

We will use the standard cycle notation of permutations~\cite{dixon1996permutation}.
Label~$n$ objects as~$a_1,\ldots,a_n$ in an arbitrary way.
Consider the transposition~$\tau=(a_{2} a_{4})$
and the $n$-cycle~$\sigma = (a_1 \cdots a_{n})$.
Then
\[
	\tau \circ \sigma = (a_{1} a_{4} \cdots a_{n}) \circ (a_{2} a_{3})
\]
is a derangement. Clearly~$\sigma^{-1}$ is also a derangement.
Therefore the transposition~$(a_{1} a_{3}) = \sigma^{-1} \circ \tau \circ \sigma$
is a product of two derangements. Since the labels~$a_1,\ldots,a_n$ are arbitrary,
this holds for any transposition (by conjugation by elements of~$\sym(n)$,
if one transposition is product of two derangement, then all transpositions are).
\end{proof}

Note that if~$f \in V_s$, only people among~$(N-M,N]$ can (and have to) move.
The next lemma shows that if~$f,g\in V_s$ agree on~$[1,N-M]$,
then there is a directed path from~$f$ to~$g$.
In the case~$M\neq 3$ the path lies in~$V_s$,
while the case~$M=3$ is exceptional and requires leaving~$V_s$ momentarily.

\begin{lemma} \label{lem:exchange}
Let~$f,g\in V_s$.
Suppose that the restrictions of~$f$ and~$g$ to~$[1,N-M]$ are equal.
Then~$g$ can be reached from~$f$ in at most~$4$ steps.
\end{lemma}
\begin{proof}
If~$M\neq 3$ this follows from Lemma~\ref{lem:derangements}:
since~$g\mid_{(N-M,N]}=\sigma \circ f_{(N-M,N]}$
for some permutation~$\sigma \in \sym([M])$, it can be reached in at most~$4$
derangements of the set~$(N-M,N]$.

Now let~$M=3$.
Since derangements of $3$~objects are $3$-cycles, it is enough to prove the
case in which~$g$ is obtained from~$f$ by a transposition. The case~$N=2$ is trivial.
Let~$N=M=3$.
Up to relabeling rooms, we can suppose~$f=(f(1),f(2),f(3))=(1,2,3)$ and~$g=(2,1,3)$.
The following is a directed path from~$f$ to~$g$:
\[
	(1,2,3), (2,3,2), (2,2,1), (2,1,3).
\]
Now let~$3=M<N$. Up to relabeling rooms and up to rotating~$(N-M,N]$,
we can suppose~$f=(\ldots,1,1,2,3)$ and~$g=(\ldots,1,2,1,3)$,
where~$\ldots$ is the same for both and will not be changed.
The following is a directed path from~$f$ to~$g$:
\[
	(\ldots,1,1,2,3), (\ldots,1,2,3,2), (\ldots,3,2,2,1), (\ldots,1,2,1,3).
\]
This concludes the proof.
\end{proof}

\begin{definition} \label{def:Vc}
Let~$V_c \subseteq V$ be the set of configurations~$f$
such that the restriction of~$f$ to~$(N-M,N]$ is constant.
\end{definition}

Intuitively, $V_c$~is the set of configurations in which the $M$~smartest people
occupy the same room:
if~$N\leq M$, this means that all people are in the same room.
While the set~$V_s$ plays the role of central hub in the graph~$G(N,M)$,
the vertices~$V_c$ are poorly connected (see also Figure~\ref{fig:1}):

\begin{lemma} \label{lem:impossible}
Let~$M\leq N$.
Then there is no edge~$(f,g)$ with~$g\in V_c$.
\end{lemma}
\begin{proof}
As in Lemma~\ref{lem:spread}, we may and will ignore~$[1,N-M]$,
and assume~$N=M$.

Let~$g(k)=r$ for all~$k\in [N]$ and
suppose that there is an edge~$(f,g)$.
The set~$f^{-1}(r)$ must be empty
(otherwise~$k=\max f^{-1}(r)$ would move to a different room, i.e.~$g(k)\neq r$).
Moreover, for every room~$s\neq r$ we must have~$\abs{f^{-1}(s)}\leq 1$
(otherwise~$k=\min f^{-1}(s)$ would remain in the same room, i.e.~~$g(k)=s\neq r$).
Since there are as many people as rooms, this contradicts the Pigeonhole Principle.
\end{proof}

The previous lemma shows that vertices in~$V_c$ have in-degree zero.
The following lemma implies that vertices in~$V\setminus V_c$
have positive in-degree, and can be reached from~$V_s$ in one step.

\begin{lemma} \label{lem:concentrate}
Let~$g\in V\setminus V_c$.
Then there is an edge~$(f,g)$ with~$f\in V_s$.
\end{lemma}
\begin{proof}
A concrete example illustrating the proof is given in Example~\ref{ex:concentrate}.
Let~$A\subseteq (N-M,N]$ be
the set of people among~$(N-M,N]$ that are the smartest in some room
according to~$g$. Note that~$\abs{A}=\abs{g(A)}\geq 2$,
and in particular there exists a derangement~$\sigma:A\to A$.
Let~$\iota$ be any bijection from~$(N-M,N]\setminus A$ to~$[M]\setminus g(A)$.
Let~$f \in V$ be defined as follows:
\[
	f(k) =
	\begin{dcases*}
	g(k) & if $k\in [1,N-M]$ \\
	g(\sigma(k)) & if $k\in A$ \\
	\iota(k) & if $k\in (N-M,N]\setminus A.$
	\end{dcases*}
\]
We claim that~$f\in V_s$.
Since~$\sigma$ is a derangement of~$A$, we have
\[
	f((N-M,N])=g(\sigma(A)) \cup \iota((N-M,N]\setminus A) = g(A) \cup ([M]\setminus g(A))
		= [M],
\]
and thus~$f\in V_s$.
We claim that~$(f,g)$ is an edge. In the configuration~$f$,
the $M$~smartest people occupy all $M$~rooms, thus people~$k\in [1,N-M]$
should not (and do not) move, i.e.~$f(k)=g(k)$.
It remains to verify that every person~$k\in (N-M,N]$, which should move,
actually does move, i.e.~$f(k)\neq g(k)$.
If~$k\in A$, we have~$f(k)=g(\sigma(k))\neq g(k)$ since the restriction~$g\vert_A$
is injective (by definition of~$A$) and~$\sigma$ is a derangement of~$A$.
If~$k\in (N-M,N]\setminus A$, we have~$f(k)=\iota(k)\in [M]\setminus g(A)$,
thus in particular~$f(k)\neq g(k)$.
\end{proof}

\begin{example} \label{ex:concentrate}
We illustrate the proof of Lemma~\ref{lem:concentrate} by a concrete example.
Consider~$N=7$ people, $M=4$ rooms~$R_1,R_2,R_3,R_4$, and the configurations
\begin{align*}
&
\begin{tabular}{ C{1cm} | C{2cm} | C{2cm} | C{2cm} | C{2cm} |}
$f:$ & 1 \ 2 \ \underline 4 & 3 \ \underline 7 & \underline 6 & \underline 5  \\
\end{tabular}
\\
&
\begin{tabular}{ C{1cm} | C{2cm} | C{2cm} | C{2cm} | C{2cm} |}
$g:$ & 1 \ 2  & 3\ \ \underline 4\ \ \underline 5\ \ \underline 6\ & \underline 7 &
\end{tabular}
\end{align*}
with~$g$ given and~$f$ constructed, by following the proof, as follows.
The $M$~smartest people are~$(N-M,N]=\{4,5,6,7\}$, highlighted in the display.
The configuration~$f$ is obtained by deranging~$\{6,7\}$ and injecting~$\{4,5\}$
into the remaining rooms. In the notation of
the proof, we have~$A=\{6,7\}$ and~$\sigma=(67)$.
We choose the function~$\iota:\{4,5\}\to \{R_1,R_4\}$ as~$\iota(4)=R_1$ and~$\iota(5)=R_4$.
From~$f$ to~$g$, the people moving are precisely
the smartest in each room, thus~$(f,g)$ is an edge.
\end{example}

\begin{proof} [Proof of Theorem~\ref{thm:1}]
First, suppose that~$M\geq N+1$. We will prove that~$G(N,M)$ is strongly connected.
Since~$M\geq N+1$, the set~$V_s$ is the subset of injective functions
and~$V_c$ the subset of constant functions.
Fix~$f,g\in V$.
By Lemma~\ref{lem:spread}, there is a directed path of length at most~$N$
from~$f$ to some~$f_1\in V_s$.
By Lemma~\ref{lem:concentrate} there is a directed path of length~$1$ 
from some~$g_1\in V_s$ to~$g$.
Since~$[1,N-M]=\emptyset$, the restrictions of~$f_1$ and~$g_1$ to~$[1,N-M]$ are equal
(they are both empty functions), thus by
Lemma~\ref{lem:exchange} there is a path of length at most~$4$ from~$f_1$ to~$g_1$.
We conclude that every~$g$ can be reached from every~$f$ in
at most~$N+5$ steps.

Conversely, suppose that~$G(N,M)$ is strongly connected.
Then there is an edge~$(f,g)$ with~$f\in V$ and~$g\in V_c$.
By Lemma~\ref{lem:impossible} this implies~$M\geq N+1$.
\end{proof}


\section{Proof of Theorem~\ref{thm:2}} \label{sec:proof_2}

In order to prove Theorem~\ref{thm:2}, it is convenient to first consider
an asynchronous variant of the game in which only one person moves at a time.

\begin{lemma} \label{lem:asynch}
Let~$N\geq 1$ and~$M\geq 2$.
Consider a variant of our game in which, at each step, a room is selected arbitrarily
and the smartest person of that room only moves to another room, while no one else moves.
Let~$G'(N,M)$ be the corresponding configuration graph.
Then the graph~$G'(N,M)$ is strongly connected.
In particular, any vertex can be reached from any other
in~$\O(N^2)$ steps.\footnote{\label{foot}
The proof of this lemma is an $\O(N^2)$-algorithm solving the asynchronous game.
For~$M=2$ rooms, this is optimal (for example,
going from~$f(k)=1$ for all~$k\in [N]$,
to~$g(k)\equiv k\, (\mathrm{mod}\, 2)$ for all~$k\in [N]$,
takes~$\Omega(N^2)$ steps).
For~$N\gg M\geq 3$, we conjecture that~$\O(N^2)$ can be improved.
Notice that Lemma~\ref{lem:asynch} is the only source of non-linearity in the paper:
any sub-quadratic solution to the asynchronous game will give,
through our proof of Theorem~\ref{thm:2}, a sub-quadratic solution to the synchronous game.
}
\end{lemma}
\begin{proof}
The statement is trivial for~$N=1$.
For~$N\geq 2$, begin by moving one at a time
all individuals from the room occupied by the person~$1$.
Once~$1$ is alone, move~$1$ to the intended room.
Once~$1$ is settled, it is never forced to move
and it does not affect the movement of~$[2,N]$.
The statement follows by induction.
\end{proof}

Now let us return to the original synchronous game.
Configurations~$f\in V_s$ do not allow elements of~$[1,N-M]$ to move.
In the next lemma, starting from~$f_1\in V_s$, we design a configuration~$g\notin V_s$
that creates just enough space
for a single element of~$[1,N-M]$ to move,
for then returning to~$f_2\in V_s$.
In this way, we are able to lift any move of the asynchronous game, played by~$[1,N-M]$,
to a path of two moves in the synchronous game, played by~$[1,N]$.

\begin{lemma} \label{lem:trick}
Let~$2\leq M\leq N$.
Let~$h_1,h_2: [1,N-M]\to [M]$. For every~$f_1\in V_s$ with~$f_1\mid_{[1,N-M]}=h_1$
there is~$f_2\in V_s$ with~$f_2\mid_{[1,N-M]}=h_2$
and a directed path in~$G(N,M)$ from~$f_1$ to~$f_2$ of length~$\O(N^2)$.
\end{lemma}
\begin{proof}
By Lemma~\ref{lem:asynch} applied to~$G'(N-M,M)$, there a path in~$G'(N-M,M)$
from~$h_1$ to~$h_2$. We show that any edge in~$G'(N-M,M)$ can be lifted to a path
in~$G(N,M)$ that starts from a specified element of~$V_s$ and ends in~$V_s$.
Therefore, up to concatenating paths, it is enough to lift a single edge of~$G'(N-M,M)$.
In other words, we can assume~$(h_1,h_2)$ of the following form:
there is a unique~$j \in [1,N-M]$ such that~$h_1(j)\neq h_2(j)$,
and moreover~$j = \max h_1^{-1}(h_1(j))$.

Let~$f_1\in V_s$ be such that~$f_1\mid_{[1,N-M]}=h_1$. Let~$j$ be the person defined in
the previous paragraph.
We are going to define~$g\in V$ and~$f_2\in V_s$
such that~$(f_1,g)$ and~$(g,f_2)$ are edges of~$G(N,M)$,
and~$f_2\mid_{[1,N-M]}=h_2$.

A concrete example illustrating the proof is given in Example~\ref{ex:trick}.

Consider the room~$r=h_1(j)$.
Since~$f_1 \in V_s$,
in the configuration~$f_1$ the person~$j$ is precisely the second smartest
person in the room~$r$. The smartest person in~$r$ is some~$k\in (N-M,N]$.
Choose any derangement~$\rho$ of~$[M]\setminus \{r\}$.
Choose any~$s \in [M]\setminus\{r\}$. Define
\[
	g(p) =
	\begin{dcases*}
	f_1(p) & if $p\in [1,N-M]$ \\
	\rho(f_1(p)) & if $p\in (N-M,N]\setminus \{k\}$ \\
	s & if $p=k$.
	\end{dcases*}
\]
From~$f_1\in V_s$ to~$g$, the people moving are exactly those in~$(N-M,N]$,
thus~$(f_1,g)$ is an edge. Moreover, notice that in the configuration~$g$
the person~$j$ (now the smartest in the room~$r$) is the only person among~$[1,N-M]$
that can (and has to) move in the next step.
Furthermore,
notice that~$r$ is the only room containing no people among~$(N-M,N]$,
and that~$s$ is the only room containing more than one (exactly two)
people among~$(N-M,N]$. In particular~$n=\min g^{-1}(s) \cap (N-M,N]$
is the only person among~$(N-M,N]$
that cannot move next step.

Choose any derangement~$\sigma$ of~$[M]$ such that~$\sigma(r)=s$. Define
\[
	f_2(p) =
	\begin{dcases*}
	h_2(p) & if $p\in [1,N-M]$ \\
	\sigma(g(p)) & if $p\in (N-M,N]\setminus \{n\}$ \\
	s & if $p=n$.
	\end{dcases*}
\]
Note that~$f_2(j)=h_2(j)$, thus~$j$ has moved to the desired room,
while no other person among~$[1,N-M)$ has moved.
The person~$p=n$ has not moved:
we have~$f_2(n) = s = g(n)$. Every
other person~$p\in (N-M,N]\setminus \{n\}$ has moved. Therefore~$(g,f_2)$ is an edge.
Moreover
\[
	f_2((N-M,N])=\sigma(g((N-M,N])) \cup \{s\}
	= \sigma([M]\setminus \{r\}) \cup \{s\}
	= [M],
\]
and thus~$f_2\in V_s$.
\end{proof}

\begin{example} \label{ex:trick}
We illustrate the proof of Lemma~\ref{lem:trick} by a concrete example.
Consider~$N=7$ people and~$M=4$ rooms~$R_1,R_2,R_3,R_4$.
Let~$f_1 = |1 2 3 4 | 5 | 6 | 7 |$, thus~$h_1 = |1 2 3|\ |\ | \ |$.
Suppose that~$h_2 = |1 2|3 |\ | \ |$, that is,
we want to move~$j=3$ from~$R_1$ to~$R_2$. Consider the configurations
\begin{align*}
&
\begin{tabular}{ C{1cm} | C{2cm} | C{2cm} | C{2cm} | C{2cm} |}
$f_1:$ & 1 \ 2 \ 3 \ \underline 4 & \underline 5 & \underline 6 & \underline 7  \\
\end{tabular}
\\
&
\begin{tabular}{ C{1cm} | C{2cm} | C{2cm} | C{2cm} | C{2cm} |}
$g:$ & 1 \ 2 \ 3 & \underline 4\ \ \underline 6\ & \underline 7\ & \underline 5\ \\
\end{tabular}
\\
&
\begin{tabular}{ C{1cm} | C{2cm} | C{2cm} | C{2cm} | C{2cm} |}
$f_2:$ & 1 \ 2 \ \underline 5 & 3\ \ \underline 4 & \underline 6 & \underline 7
\end{tabular}
\end{align*}
where~$g$ and~$f_2$ are constructed as in the proof of Lemma~\ref{lem:trick}
as follows.
In the notation of the proof,
we have~$(N-M,N]=\{4,5,6,7\}$,
$j=3$, $k=4$, $r=R_1$, $s=R_2$, $\rho=(R_2 R_4 R_3)$,
$\sigma = (R_1 R_2 R_3 R_4)$, and~$n=4$. In the configuration~$g$ we create just enough space
for a single element~$j\in [1,N-M]$ to move (here~$j=3$).
\end{example}

\begin{proof} [Proof of Theorem~\ref{thm:2}]
Let~$3\leq M\leq N$. Lemma~\ref{lem:impossible} shows that every~$g\in V_c$ has
in-degree zero. Thus every~$g\in V_c$ forms strongly connected component of its own.
Clearly any~$g\in V_c$ (in fact, any~$g\in V$) has positive out-degree.
It remains to prove that~$V\setminus V_c$ is strongly connected.

Let~$f\in V$ and~$g\in V\setminus V_c$. We will
prove that there is a directed path from~$f$ to~$g$ of length at most~$\O(N^2)$.
By Lemma~\ref{lem:spread}, there is a directed path of length at most~$N$
from~$f$ to some~$f_1\in V_s$.
By Lemma~\ref{lem:concentrate}, there is a directed path of length~$1$
from some~$g_1\in V_s$ to~$g$.
In contrast to the proof of Theorem~\ref{thm:1}, we cannot apply Lemma~\ref{lem:exchange}, yet:
since~$M\leq N$ the restrictions~$h_1 = f_1\mid_{[1,N-M]}$ and~$h_2 = g_1\mid_{[1,N-M]}$
might differ.
However, by Lemma~\ref{lem:trick}, there is~$f_2\in V_s$ such that~$f_2\mid_{[1,N-M]} = h_2$
and a path from~$f_1$ to~$f_2$ of length~$\O(N^2)$.
By Lemma~\ref{lem:exchange}, which now can be applied, we obtain a path from~$f_2$ and~$g_1$
of length at most~$4$.
We conclude that every~$g\in V\setminus V_c$ can be reached from any~$f\in V\setminus V_c$
in~$\O(N^2)$ steps.
\end{proof}

\section{Acknowledgements}
I would like to thank the anonymous referee for their valuable suggestions,
and for pointing out that Lemma~\ref{lem:asynch} has quadratic complexity.
I would also like to thank my friend Eric Sandin Vidal
for pointing out the similarity to the Towers of Hanoi.
This paper was written while I was a PhD candidate at Vrije Universiteit
Amsterdam; I am thankful to the institution for its support.


\bibliographystyle{plain}
\bibliography{refs}

\end{document}